%% file: paper_r2.tex
\documentclass[graybox]{svmult}
\input{macros}

\usepackage{dune}

\usepackage{type1cm}        
\usepackage{makeidx}         
\usepackage{graphicx}        
\usepackage{wrapfig}         
\usepackage{multicol}        
\usepackage[bottom]{footmisc}

\usepackage{newtxtext}       %
\usepackage[varvw]{newtxmath}       

\usepackage{xcolor}
\usepackage[pdftex]{hyperref}
\definecolor{hrefcolor}{rgb}{0.0,0.0,0.8}
\newcommand{\linkcolor}{hrefcolor}

\usepackage{hyperref}
\usepackage[left]{lineno}
\hypersetup{pdftitle={Structure Preserving Schemes for Cahn-Hilliard},
            pdfborderstyle={/S/U/W 1},
            linkbordercolor=\linkcolor,
            urlbordercolor=\linkcolor,
            citebordercolor=\linkcolor,
            runbordercolor=\linkcolor,
            menubordercolor=\linkcolor,
            filebordercolor=\linkcolor,
            runcolor=\linkcolor}

\makeindex             

\begin{document}

\title*{On $hp$-adaptive Structure-Preservation for the Cahn--Hilliard--Navier--Stokes Equations with Degenerate Mobility
}
\titlerunning{$hp$-adapt Structure-Preserving Methods for the CHNS Equations}
\author{Jimmy Kornelije Gunnarsson\orcidID{0009-0000-7610-6360} and\\ Robert Kl\"ofkorn\orcidID{0000-0001-9664-0333}}
\institute{Jimmy Kornelije Gunnarsson \at Centre for Mathematical Sciences, Lund University, Box 117, 22100 Lund, Sweden \email{jimmy\_kornelije.gunnarsson@math.lu.se}
\and Robert Kl\"ofkorn \at Centre for Mathematical Sciences, Lund University, Box 117, 22100 Lund, Sweden \email{robertk@math.lu.se}}
%
%
\maketitle
\input{abstract}
\section{Introduction}
\label{sec:1}
\input{introduction}
\section{Discretization}
\label{sec:2}
\input{discretization.tex}
\section{Numerics}\label{sec:numerics}
\input{numerics}
\section{Summary and Outlook}\label{sec:summary}
\input{summary}


%
%

\bibliographystyle{spmpsci}
\bibliography{refs}
\end{document}

%% file: macros.tex
\newcommand{\sipg}{SIPG\xspace}
\newcommand{\swip}{SWIP\xspace}
\newcommand{\sipgL}{SIPG-L\xspace}
\newcommand{\swipL}{SWIP-L\xspace}
\newcommand{\sipgDL}{SIPGD-L\xspace}
\newcommand{\swipDL}{SWIPD-L\xspace}
\newcommand{\swipD}{SWIPD\xspace}
\newcommand{\sipgD}{SIPGD\xspace}
\newcommand{\fem}{FEM\xspace}
\newcommand{\femL}{FEM-L\xspace}

\newcommand{\grid}{{\mathcal{T}_h}}
\newcommand{\entity}{K}
\newcommand{\elem}{\entity}
\newcommand{\elemin}{\elem^{-}}
\newcommand{\elemout}{\elem^{+}}

\newcommand{\isec}{e}

\newcommand{\RRR}{\mathbb{R}}

\newcommand{\R}{\RRR}

\newcommand{\nbold}{\boldsymbol{n}}

\newcommand{\jump}[1]{[ \! [ {#1} ] \! ] }
\newcommand{\vjump}[1]{ [ {#1} ] }
\newcommand{\aver}[1]{\{ {#1} \} }
\newcommand{\haver}[1]{\langle {#1} \rangle}

\newcommand{\chem}{\upsilon}

\newcommand{\Kplus}{{K^+_e}}
\newcommand{\Kminus}{{K^-_e}}
\newcommand{\varphiplus}{\varphi_{|\Kplus}}
\newcommand{\varphiminus}{\varphi_{|\Kminus}}

\newcommand{\nboldplus}{\nbold_{\Kplus}}

\newcommand{\pf}[0]{\psi}

\newcommand{\mb}[1]{\boldsymbol{#1}}

\newcommand{\ome}[0]{\Omega}
\newcommand{\mbu}[0]{\mb{u}}
\newcommand{\mbx}[0]{\mb{x}}

\newcommand{\vp}[0]{\varphi}

\newcommand{\Po}[0]{\mathbb{P}}

\newcommand{\innerProd}[2]{\left\langle #1, #2\right\rangle}

\newcommand{\pmax}{{p_{\max}}}
\newcommand{\pmin}{{p_{\min}}}

%% file: abstract.tex
\abstract{We develop structure-preserving discontinuous Galerkin methods for the Cahn-Hilliard-Navier-Stokes equations with degenerate mobility. The proposed \swipDL and \sipgDL methods incorporate parametrized mobility fluxes with edge-wise mobility treatments for enhanced coercivity-stability control. We prove coercivity for the generalized trilinear form and demonstrate optimal convergence rates while preserving mass conservation, energy dissipation, and the discrete maximum principle. Comparisons with existing \sipgL and \swipL methods confirm similar stability. Validation on $hp$-adaptive meshes for both standalone Cahn-Hilliard and coupled systems shows significant computational savings without accuracy loss.}

%% file: introduction.tex
Consider the spatial domain $\Omega \subset \R^d$, $d = 2,3$, with Lipschitz continuous
boundary $\partial \Omega$ and the time interval $(0,T]$ for $T \in \R^+$. The Cahn-Hilliard
(CH) equations model phase separation in binary mixtures for a phase-field $\pf :
\Omega \times (0,T] \to [-1,1]$ following:
\begin{eqnarray}\label{eq:CH}
\partial_t \pf + \nabla \cdot (\pf \mathbf{u}) -  \nabla \cdot Pe^{-1}(M(\pf) \nabla
\chem) &=&0,\\
\chem  +Cn^2 \Delta \pf - W'(\pf) &=& 0,
\end{eqnarray}
where $Pe > 0$ is the Peclet number, $\mbu$ is an advection field, $\chem$ is the
chemical potential, $W(\pf) := \frac{1}{4}(1 - \pf^2)^2$ is the double-well potential,
and $M(\pf) := \max{\{1 - \pf^2,0\}}$ is the mobility function. We equip Eq.~\eqref{eq:CH}
with homogeneous Neumann boundary conditions $\mb{n} \cdot \nabla \pf = 0$ and $\mb{n}
\cdot \nabla \chem = 0$ on $\partial \Omega$.

We present below a special case of a fundamental theorem for boundedness.
\begin{theorem}[Boundedness \cite{Elliott:2000}]\label{thm:bounded}
Suppose that $\Omega$ is convex and that the initial phase-field satisfies $\pf^0 \in H^1(\Omega)$
and $||\pf^0||_{L^\infty(\Omega)} \leq 1$. Then, if the mobility function $M(\pf)$ is
defined as above, and $\int_{\Omega} M(\pf^0) + W(\pf^0) < C$ for $C > 0$, then the weak solution $\pf \in H^1(\Omega)$
of the CH equation satisfies a weak maximum principle, i.e., $||\pf||_{L^\infty(\Omega)}
\leq 1$ for $t \in (0,T]$.
\end{theorem}

While Theorem~\ref{thm:bounded}, proven in \cite[Thm. 1]{Elliott:2000}, guarantees
a weak maximum principle at the continuum level for a degenerate mobility, achieving
discrete boundedness remains a significant challenge. The standard Finite Element Method (\fem),
symmetric weighted interior penalty (\swip), and symmetric interior penalty Galerkin
(\sipg) discretizations fail to preserve boundedness without additional stabilization
\cite{Gunnarsson:2026}, even when using the degenerate mobility $M$. Recent work,
however, has demonstrated that boundedness can be numerically achieved through carefully
designed limited Galerkin schemes, including \femL for \fem, and for Discontinuous Galerkin (DG) schemes with \sipgL, and
\swipL \cite{Gunnarsson:2026} with the degenerate CH equations.

In this work, we extend these structure-preserving methods by introducing new mobility
fluxes using the DG formulations with intersection-wise
mobility treatments. These extend the \swipL and \sipgL methods introduced in \cite{Gunnarsson:2026},
by allowing for better control of the coercivity-stability balance. \par

Boundedness of the CH equations is in particular important when coupled to the incompressible
Navier-Stokes (NS) equations:
\begin{eqnarray}
    \rho(\pf)\big(\partial_{t} \mbu + \mbu \cdot \nabla \mbu\big)
    + \mb{J} \cdot \nabla \mbu
    + \nabla \cdot \big(P \mathbb{I} - 2\mu(\pf)D(\mbu)\big)
&=& -\frac{1}{WeCn} \pf \nabla \chem
    \label{eq:ns1nd} \\
  \nabla \cdot \mathbf{u} &=& 0
    \label{eq:ns2nd}
\end{eqnarray}
where $\rho(\pf) := \frac{1+\pf}{2}\rho_1 + \frac{1-\pf}{2}\rho_2$ is the density,
$\mb{J} := \frac{\rho_2 - \rho_1}{2}Pe^{-1}M(\pf)\nabla \chem$ is the mass flux,
$P$ is the pressure, $\mu(\pf) := \frac{1+\pf}{2}\mu_1 + \frac{1-\pf}{2}\mu_2$ is
the viscosity, and $D(\mbu) = \frac{1}{2}(\nabla \mbu + \nabla \mbu^T)$ is the rate
of deformation tensor, where $\mathbb{I}$ denotes the identity tensor. Moreover,
$\rho_j$ and $\mu_j$ are the characteristic densities and viscosities for phases
$j = 1, 2$, respectively. We equip the boundary conditions $\mathbf{n} \cdot \mbu|_{\partial
\Omega} = \mathbf{n} \cdot \nabla P|_{\partial \Omega} = 0$. We note that the expected
bounds for the physical mass density $\rho \in [\rho_1, \rho_2]$ follow directly
from boundedness of the phase-field $\pf \in [-1,1]$.\par

The main contributions of this paper are: (i) the development of \swipDL and \sipgDL
methods with a parametrized mobility flux for enhanced stability, (ii) a proof of
coercivity for the generalized trilinear form for \swipD and \sipgD schemes, and
(iii) numerical validation demonstrating optimal convergence rates and structure
preservation for both standalone CH and coupled CHNS systems. We compare the performance
of \sipgL, \swipL, \sipgDL, and \swipDL methods through extensive numerical experiments,
including $hp$-adaptive formulations for our new schemes.

%% file: discretization.tex
In this section, we consider the spatial and temporal discretization of the CHNS
equation system. Our main focus is to motivate the design of the \swipDL and \sipgDL
methods and to provide a proof of coercivity for the resulting trilinear form, while
the remaining details follow a similar procedure as in \cite{Gunnarsson:2026}. \par

 Let the spatial domain $\ome$ \xspace be partitioned into a union of $N_\grid$ non-intersecting
elements $\elem$ forming a mesh $\grid = \cup_{i = 1}^{N_\grid} \elem_i$. Then we denote
by $\Gamma_i$ with unit normal $\mb{n}$ the set of all intersections between
two
elements of the grid $\grid$, and the set of all
intersections, also with the boundary of the domain $\Omega$, is denoted by
$\Gamma$.
For an intersection $\isec \in \Gamma$ we denote the adjacent elements with $\elemin_\isec$
and
$\elemout_\isec$ ($\elemin_\isec = \elemout_\isec$ for $\isec \in \Gamma \setminus \Gamma_i$)
and
define
\begin{equation}
  h_\isec := \frac{2|\elemin_\isec||\elemout_\isec|}{|\isec|(|\elemin_\isec| + |\elemout_\isec|)}
\quad \forall
\isec \in \Gamma.
\end{equation}
The global mesh width is then defined as $h = \max_{\isec \in \Gamma} h_\isec$.
For the time discretization we consider a uniform partition of the time interval
$[0,T]$ with time increment $\dt = \frac{T}{N}$ for some $N \in \mathbb{N}$ and denote
the time levels with $t_n = n \cdot \dt$ for $n = 0,1,\ldots,N$. For brevity we will
also denote time-dependent variables with a subscript $n$, i.e., $\pf^n = \pf(\cdot,
t_n)$.

Following standard \fem notation we consider a general order \fem formulation
for the function space of trial and test functions:
\begin{equation}
V_h^p = \{ \varphi \in L^2(\grid) : \varphi|_\elem \in \Po^p(\elem), \forall \elem
\in \grid
\},
\end{equation}
where $\Po^p(\elem)$ denotes a polynomial space of order at most $p$ on an
element $\elem$. We also permit $p$-adaptivity over the mesh $\grid$, and as such,
the
polynomial order $p$ may vary across elements $\elem$, and we denote the local polynomial
order as $p_\elem$. We denote the maximal permitted polynomial order as $p_{\max}$,
and the minimal order as $p_{\min}$. We describe the application further in Sec.~\ref{sec:hp}.
\par
 Before proceeding, we introduce
operators $\vjump{\cdot}, \aver{\cdot}$ and $\haver{\cdot}$ for
 $\isec \in \Gamma_i$ as
\begin{equation*}
  \begin{split}
    \vjump{\varphi} = \varphiminus  - \varphiplus, \qquad
    \aver{\varphi} = \frac{1}{2}\left(\varphiminus+\varphiplus\right), \qquad
\haver{\varphi} = \frac{2\, \varphiplus \, \varphiminus}{\varphiplus + \varphiminus},
\\
  \end{split}
\end{equation*}
for some $\varphi$, where $\aver{\cdot}$ and $\haver{\cdot}$ denote the arithmetic
and harmonic averages, respectively.
To simplify the notation $\varphi^\pm := \varphi_{|K_e^\pm}$ is also used
frequently.
Moreover, we introduce the subscript $\varphi_\oplus := \max{\{0,\varphi\}}$
and $\varphi_\ominus := \min{\{0, \varphi\}}$ to denote the positive and negative
restriction of a function, respectively. This notation will in particular be
utilized for upwinding.  \par
\subsection{Discontinuous Galerkin}
Given a degenerate mobility function $M: V_h^p \to [0,1]$, we introduce the trilinear
form $b: L^\infty(\grid) \times V_h^p \times V_h^p \to \R$:
\begin{eqnarray} \label{eq:mobilityDG}
b(M(\pf_h), \chem_h, \vp) &=
    \int_\grid M(\pf_h)\, \nabla \chem_h \cdot \nabla \vp\, dx
    + \sum_{e \in \Gamma} \int_e
        \frac{\pena_e \Lambda_\isec(M(\pf_h))}{h_\isec} \vjump{\chem_h} \vjump{\vp}
\notag \\
        &- F(M(\pf_h), \chem_h)\, \vjump{\vp}
        -  F(M(\pf_h), \vp)\, \vjump{\chem_h}
     ds,
    \label{eq:diffmob}
\end{eqnarray}
where $\eta_\isec \geq 0$ is a penalty parameter over $\isec$, $F$ is the mobility
flux, and $\Lambda_\isec(M(\pf_h))$ is a mobility flux function. We examine the coercivity
properties arising from the degeneracy of $M$. We consider the following formulations
for $F$:
\begin{equation}
F(M(\pf_h), \chem_h) = \haver{M(\pf_h)} \aver{ \nabla \chem_h}, \quad F(M(\pf_h),
\chem_h) = \aver{M(\pf_h) \nabla \chem_h},
\end{equation}
corresponding to a weighted average with a harmonic average for the mobility flux
$\haver{M(\pf_h)}$ and an arithmetic average $\aver{M(\pf_h) \nabla \chem_h}$, which
correspond to the \swip and \sipg discretization, respectively. We note that the
\swip formulation with $\haver{M(\pf_h)}$ follows from choosing the weights $w^\pm
= \frac{M(\pf_h)^\mp}{M(\pf_h)^+ + M(\pf_h)^-}$, as suggested in \cite{ern:2008}
for treating a diffusion-type formulation for DG.
\begin{lemma}[Trace inequality, \cite{Riviere:2008}]\label{lem:trace}
Consider an intersection $e \in \Gamma_i$ shared by two elements $\elemin_\isec$
and
$\elemout_\isec$.
Then
  there exists a constant $C_{\elem_\isec^\pm} > 0$ independent of the mesh
width $h$ such that
for all $\varphi_h \in V^p_h$ for $p > 0$ and $\isec \in \Gamma_i$ the following
inequality
holds:
\begin{equation}\label{eq:traceineq}
  ||\nabla \varphi_h^\pm \cdot \mathbf{n}^+||_{L^2(\isec)}^2 \leq \frac{|\isec|}{|\elem_\isec^\pm|}
C_{\elem_\isec^\pm},
||\nabla
\varphi_h||_{L^2(K_\isec^\pm)}^2,
\end{equation}
where, in the special case of $K_\isec$ being a quadrilateral or triangle, $C_{\elem_\isec^\pm}
= \frac{p(p+d-1)}{d}$ (see further treatments and formulations in \cite{Ainsworth:2009}
and \cite{Epshteyn:2007}).
\end{lemma}
\begin{lemma}[Trace inequality for mixed polynomial order]\label{lem:trace2}
  Under the same setting as in Lemma~\ref{lem:trace} and for $\varphi_h \in V_h^\pmax$
for $\pmax > 0$.
the following inequality holds for all $\isec \in \Gamma_i$:
\begin{equation}
    ||\aver{\nabla \varphi \cdot \mathbf{n}}||_{L^2(\isec)}^2 \leq \frac{C_{\isec}}{2h_e}
\left( ||\nabla \varphi^+ \cdot \mathbf{n}^+||_{L^2(K_\isec^+)}^2
+ ||\nabla \varphi^- \cdot \mathbf{n}^+||_{L^2(K_\isec^-)}^2 \right),
\end{equation}
where for the special case where $K^\pm_\isec$ are quadrilaterals or triangles:
\begin{equation}
C_{\isec} =  \underset{{p \in \{p_{\elem^+_\isec}, p_{\elem^-_\isec}\}}}{\max} \frac{p(p+d-1)}{d}.
\end{equation}
\end{lemma}
\begin{proof}
The proof follows from the triangle inequality and the trace inequality in Lemma~\ref{lem:trace}:
\begin{equation}
    ||\aver{\nabla \varphi \cdot \mathbf{n}}||_{L^2(\isec)}^2 \leq \frac{1}{2} \left(||\nabla
\varphi^+ \cdot \mathbf{n}||_{L^2(\isec)}^2 + ||\nabla \varphi^- \cdot \mathbf{n}||_{L^2(\isec)}^2\right),
\end{equation}
then upon applying the trace inequality in Lemma~\ref{lem:trace} we obtain the desired
result by then maximizing the quantity with
\begin{equation}
    \frac{|\isec|}{|\elem_\isec^+|} C_{\elem_\isec^+} + \frac{|\isec|}{|\elem_\isec^-|}
C_{\elem_\isec^-}
\leq 2\frac{C_e}{h_e},
\end{equation}
for a joint inequality.
\end{proof}

\begin{theorem}[Coercivity] \label{thm:coercivity}
Consider a fixed phase-field $\pf_h \in V_h^\pmax$ for $\pmax > 0$ satisfying $M(||\pf_h||_{L^\infty(\grid)})
\geq \delta$ for some small $\delta > 0$. Then for $\upsilon_h \in V_h^\pmax$ there exists
a constant $\pena_\isec > 0$ and mobility flux $\Lambda_e$ independent
of the mesh width $h$ such that for $\isec \in \Gamma_i$ the trilinear form $b$ in
Eq.~\eqref{eq:diffmob} is coercive such that for the DG semi-norm
\begin{equation*}
  ||\upsilon_h||_{DG}^2 = ||\sqrt{M(\pf_h)} \nabla \upsilon_h||^2 + \sum_{\isec \in
\Gamma_i} \int_\isec \frac{\Lambda_\isec(M(\pf_h)) \eta_\isec}{h_\isec} \vjump{\upsilon_h}^2
ds,
\end{equation*}
for some  $C > 0$ where $ ||\upsilon_h||_{DG}^2 \leq C \, b(M(\pf_h),
\upsilon_h, \upsilon_h)$. Then the following coercivity condition holds:
\begin{equation}
    \pena_\isec \geq \max_{p \in \{\elem_\isec^+, \elem_\isec^-\}}\max\{p(p+d-1), 1\}, \quad
\Lambda_\isec(M(\pf_h)) \geq \tilde{M}_e^{2\alpha} (\pf_h)
\lambda^\star,
\end{equation}
where for $\alpha \in [0,\frac{1}{2}]$, the maximum contrast $\lambda^\star$ is defined
as
\begin{equation}
    \lambda^\star = \max_{\elem \in \grid} \sum_{\isec \in \partial \elem}  \frac{||\tilde{M}_e
(\pf_h)||_{L^\infty(e)}^{2 -2\alpha}}{\underset{\mbx \in \elem}{\min} M(\pf_h(\mbx))},
\end{equation}
and the flux $\tilde{M}_\isec(\pf_h)$ is defined as
\begin{equation}
\tilde{M}_\isec(\pf_h) = \haver{M(\pf_h)}, \quad \tilde{M}_\isec(\pf_h) = \max \{M(\pf_h^+),
M(\pf_h^-)\},
\end{equation}
for the \swip and \sipg mobility fluxes, respectively.
\end{theorem}
\begin{proof}
The overall proof follows similarly to the one in \cite[Theorem 3.2]{Gunnarsson:2026},
with the
main difference being the treatment of the mobility term $M(\pf_h)$ for the first
Cauchy-Schwarz inequality over the intersections $\isec \in \Gamma_i$. In particular
we note that for $\isec \in \Gamma_i$ we
have
\begin{equation}
  | F(M(\pf_h), \chem_h)\jump{\chem_h}| \leq  ||\tilde{M}(\pf_h)^{1-\alpha} \aver{\nabla
\chem \cdot \nboldplus}||_{L^2(e)} ||\tilde{M}(\pf_h)^\alpha \jump{\chem_h}||_{L^2(e)}.
\end{equation}
Following a similar approach to the proof in \cite{Gunnarsson:2026}, which
considered $\alpha = 0$, we generalize to a free variable $\alpha \in [0,\frac{1}{2}]$
to obtain
the desired result by also applying the modified trace inequality in Lemma~\ref{lem:trace2}
and maximizing the resulting quantity for a joint inequality over the intersection
$\isec$.
\end{proof}
For the remainder of this study we will consider the regularization $M_\delta(\pf_h)
= \max\{M(\pf_h), \delta\}$ for some small $\delta > 0$ to ensure the coercivity
condition in Theorem~\ref{thm:coercivity} is satisfied, and we will simply denote
this regularized mobility as $M(\pf_h)$ for brevity and we pick $\delta = 10^{-20}$
(practically $0$) for the numerical simulations.
\begin{remark}[mobility  flux]\label{rem:fs}
It is complicated to estimate the maximum contrast $\lambda^\star$ for a sharp estimate.
For simplicity, it is sufficient to consider the mobility flux as $\Lambda_e = \beta
\tilde{M}_e^{2\alpha}(\pf_h)$
for some constant a-priori estimate $\beta \geq \lambda^\star$. For instance,
 $\beta = 5$ was considered in the numerical experiments in \cite{Gunnarsson:2026}.
\end{remark}
\begin{corollary}[Case $V_h^0$]\label{cor:coercivity}
For the case of \xspace$V_h^0$ the coercivity condition in Thm.~\ref{thm:coercivity} reduces
to
\begin{equation}
    \pena_\isec \geq 1, \quad \Lambda_\isec(M(\pf_h)) \geq \tilde{M}(\pf_h).
\end{equation}
\end{corollary}
\begin{proof}
Since for $\vp \in V_h^0$ then $\nabla \vp = 0$ and thus only the penalty term remains
in the trilinear form $b$ in Eq.~\eqref{eq:diffmob}, and consequently the coercivity
condition reduces to $\pena_\isec \Lambda_\isec(M(\pf_h)) \geq \tilde{M}(\pf_h)$,
which is satisfied for the stated conditions.
\end{proof}
Similarly to Eq.~\eqref{eq:mobilityDG} we introduce the following bilinear form
for the discretization of the Laplacian operator in Eq.~\eqref{eq:CH}:
\begin{eqnarray}
a(\pf_h, \vp) &=
    \int_\grid \nabla \pf_h \cdot \nabla \vp\, dx
    + \sum_{e \in \Gamma} \int_e
        \frac{\pena_e^\Delta}{h_\isec} \vjump{\pf_h} \vjump{\vp} \notag \\
        &- \aver{\nabla \pf_h \cdot \nbold} \vjump{\vp}
        - \aver{\nabla \vp \cdot \nbold} \vjump{\pf_h}
     ds,
    \label{eq:laplace}
\end{eqnarray}
where $\pena_e^\Delta \geq \pena_e$ is a local penalty parameter for the Laplacian
operator. We introduce the superscript $\Delta$ to distinguish the penalty parameter
for the Laplacian from the one for the mobility term in Eq.~\eqref{eq:mobilityDG},
and in particular, to lump the constants in practice so that $\pena_e^\Delta = \beta
\pena_e$ for some $\beta \geq 1$ following Rmk.~\ref{rem:fs}. \par
When we consider advection coupling we introduce the standard upwinding for the advection
term:
\begin{equation}\label{eq:adv2}
c(\mb{u}, \pf_h, \vp) = \int_\grid \mb{u} \cdot \nabla \vp \, \pf_h \, dx -
\sum_{e \in \Gamma_i} \int_e \left( \aver{\mb{u} \cdot \mathbf{n}^+}_{\oplus}
\pf_h^+ + \aver{\mb{u} \cdot \mathbf{n}^+}_{\ominus} \pf_h^-  \right) \vjump{\vp}
ds,
\end{equation}
Thus, we have introduced all the DG-specific bilinear and trilinear forms for the
spatial discretization of the CH equations. In the following, we consider the time
discretization and coupling to the NS equations. \par
To treat the non-linear term $W$ we introduce the Eyre splitting:
\begin{equation}
\Phi^+(\pf_h) = \pf_h^3, \quad \Phi^-(\pf_h) = -\pf_h,
\end{equation}
where clearly $W'(\pf_h) = \Phi^+(\pf_h) + \Phi^-(\pf_h)$. Below we present our proposed
schemes for the CH equations:
\begin{eqnarray}
  \frac{1}{\dt}\innerProd{\pf_h^{n+1}- \pf_h^{n}}{\vp} -
c(\tilde{\mb{u}}^{n + \frac{1}{2}}, \pf_h^{n}, \vp) + Pe^{-1} b(M(\pf_h^{n+1}), \chem_h^{n+1},
\vp) &=& 0,  \label{eq:ch1stt}\\
  \innerProd{\chem_h^{n+1}}{\xi} - \innerProd{\Phi^+(\pf_h^{n+1}) + \Phi^-(\pf_h^{n})}{\xi}
- Cn^2 a (\pf_h^{n+1}, \xi)  &=& 0, \label{eq:ch2ndt}
\end{eqnarray}
for $\xi, \vp \in V_h^p$ and $\tilde{\mb{u}}^{n+ \frac{1}{2}}$ being the velocity field
at time step $n + \frac{1}{2}$. We note that the mobility term, and penalty weights,
are treated implicitly in time. The main difference between the proposed schemes
is the mobility flux for the penalty term, where we consider both a harmonic average
and an intersection maximum as described in Eq.~\eqref{eq:mobilityDG}. The outline
and algorithm of the scheme are very similar to those detailed in \cite{Gunnarsson:2026},
as
it is simply an algebraic reformulation of the same scheme with a different mobility
flux for the penalty term, therefore we do not repeat the details here. \par
For the NS coupling we follow exactly the same formulation as in~\cite{Gunnarsson:2026}
with a post-processing step. Our modification includes projecting the \fem velocity
field $\mbu$ to a Brezzi-Douglas-Marini space of first order to obtain a divergence-free
velocity field $\tilde{\mbu}$.
\begin{definition}[Phase-field mass]\label{def:mass}
We refer to the quantity
\begin{equation}
  m_{\pf_h} = \frac{1}{|\grid|} \int_\grid \pf_h \, dx
\end{equation}
with $|\grid| = \int_\grid 1 \, dx$ as the phase-field mass.
\end{definition}
\begin{definition}[Discrete energy]\label{def:energy}
We refer to the quantity:
\begin{equation}
  \mathcal{E}[\pf_h, \mbu] = \frac{1}{WeCn}\int_\grid W(\pf_h) + \frac{Cn^2}{2} a(\pf_h,
\pf_h) \, dx + \int_\grid \rho(\pf_h) \frac{|\mbu|^2}{2} \, dx,
\end{equation}
as the discrete energy.
\end{definition}
Following similar arguments as in \cite{Gunnarsson:2026} and references therein,
the provided schemes are energy dissipative and mass conservative on the discrete
level for the quantities defined in Definitions~\ref{def:mass} and \ref{def:energy}.
\subsection{Limiter and $hp$-adaptivity} \label{sec:hp}
\begin{figure}[ht]
\includegraphics[width=0.89\textwidth]{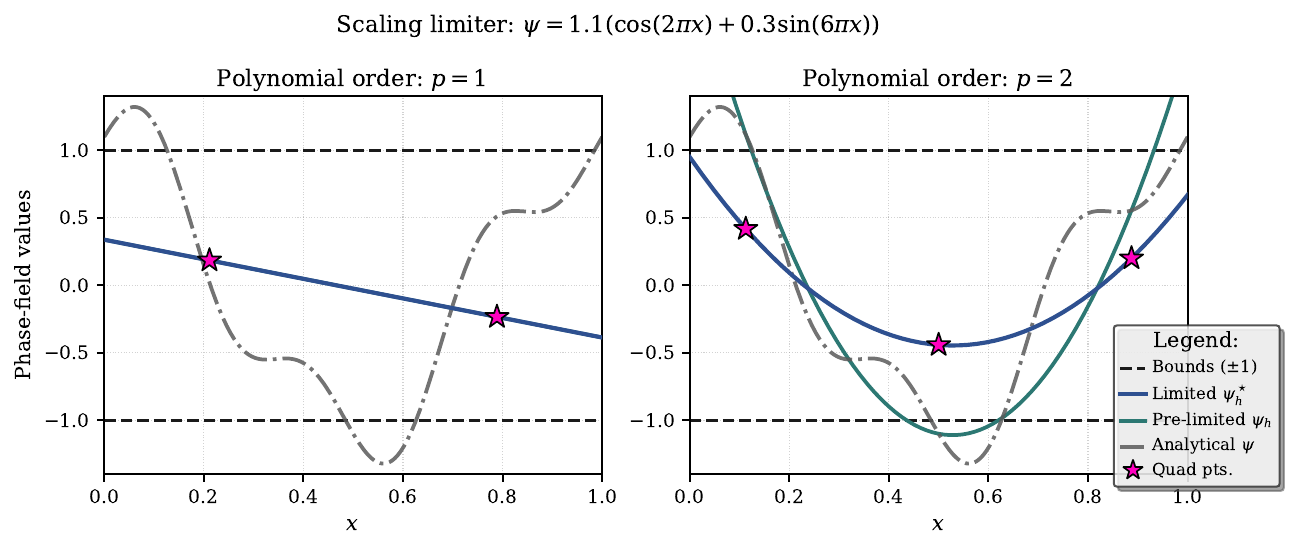} \label{fig:limiter}
\caption{Example of a scaling limiter for local $\Po^2$ phase-field $\pf_h$}
\end{figure}
A Zhang-Shu scaling limiter \cite{Zhang:2010} is applied to $\pf_h$, with quadrature
evaluation for local polynomial order $p_{\elem} \geq 1$. The limiter structure is
outlined in Fig.~\ref{fig:limiter} and described in \cite{Gunnarsson:2026}. Following
the notation in \cite{Gunnarsson:2026} we denote by \swipDL and \sipgDL the methods
corresponding to the mobility flux with $\alpha = \frac{1}{2}$ in Theorem~\ref{thm:coercivity},
while the methods with $\alpha = 0$ are denoted by \swipL and \sipgL, respectively.
We note that the mobility flux with $\alpha = \frac{1}{2}$ corresponds to a more
diffusive flux, which facilitates treatment over the intersections $\isec$, at the
potential cost of a larger global contrast $\lambda^\star$. \par
For $h$-adaptivity we follow the formulation provided in \cite{Gunnarsson:2026} and
we consider an indicator function $H(\pf_h) := \frac{(1 - \pf^2)}{4}$ and refinement
for elements $K$ where $H(\pf_h) < 0.0525$, and coarsening is suggested when
$H(\pf_h) > 0.15$. Meanwhile, for $p$-adaptivity, coarsening is performed for elements
$K$ where $H(\pf_h) > 0.075$ and refinement follows similarly to $h$-adaptivity.
Fig.~\ref{fig:hpadapt} shows an example of $hp$-adaptivity. We note that the thresholds
for refinement and coarsening are somewhat arbitrary and can be tuned.
\begin{wrapfigure}{r}{0.45\textwidth}
\centering
\includegraphics[width=0.38\textwidth]{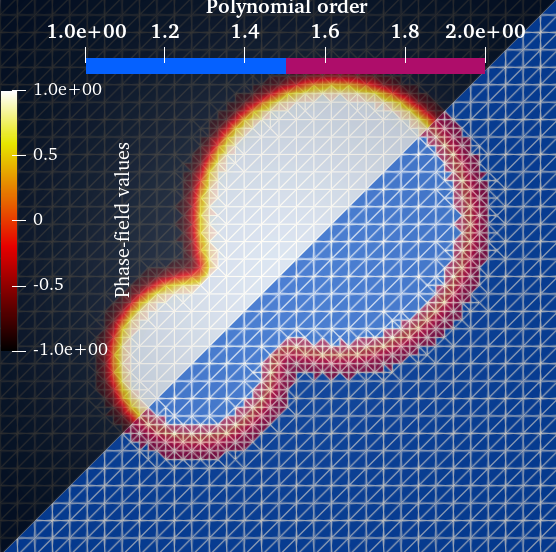}
\caption{Ex.~\ref{ex:stationary} with $hp$-adaptivity for $p_{\min} = 1$ and $p_{\max}
= 2$ at time $t = 0$.}\label{fig:hpadapt}
\vspace{-2cm}
\end{wrapfigure}
In particular, we aim to isolate elements $K$ with low-order polynomials that satisfy
$||\pf_h||_{L^\infty(K)} \lessapprox 1$ to ensure that coercivity is not broken by
populating elements with low-order polynomials with values close to the bounds, which
can ideally lead to a smaller contrast $\lambda^\star$ and less penalization in the
scheme. Such effects are examined in Section~\ref{sec:numerics}, where we compare
the performance of $h$-adaptivity and $hp$-adaptivity for the proposed schemes. Moreover,
to allow $p_{\min} = 0$ we consider $\eta_e = \underset{p \in \{p_{\elem^+_\isec},
p_{\elem^-_\isec}\}}{\max}
\max\left\{p(p+d-1), 1\right\}$.

%% file: numerics.tex
In this section, we present numerical results for the proposed schemes for both the
standalone CH equations and the coupled CH-NS system. We first consider Ex.~\ref{ex:trig}
to verify the optimal convergence rates in 2D and 3D. We then investigate the performance
of the proposed schemes for merging droplets in Ex.~\ref{ex:droplets}, and finally
consider
the rotating droplet configuration in Ex.~\ref{ex:stationary} for the coupled CH-NS
system to demonstrate
the structure-preserving properties and $hp$-adaptivity. The numerical simulations
are implemented in \dunefem~\cite{dunereview:21, dunefemdg:21}
with its \code{Python} interface~\cite{Dedner:2020}. The code to run the studied experiments is available at our GitLab repository~\cite{GunnarssonGit:2026}.
\begin{example}[Artificial trigonometric solution]\label{ex:trig}
Consider the domain $\Omega = [0,1]^d$ and the following artificial solution to the
CH equations $\pf(\mbx,0) = A\prod_{i=1}^{d} \cos(4\pi \mbx_i)$
where $A \in (0,1]$ is an amplitude parameter and $\mbx_i$ is the $i$-th component
of the spatial coordinate $\mbx$. We then consider the following forcing term $S(\mbx,t)
= Pe^{-1}\nabla \cdot (M(\pf) \nabla(Cn^2 \Delta \pf - W'(\pf)))$ added to the right-hand
side of Eq.~\eqref{eq:ch1stt},
where $Cn = 0.1$, $Pe = 0.3$ are the Cahn and Peclet numbers respectively. The simulation
is run for $t \leq T = 0.01$.
\end{example}
\begin{table}[htbp]
  \centering
  \caption{Ex.~\ref{ex:trig}: $L^2$ and $H^1$ errors and EOC for $A = 0.1$ and $A
= 0.6$ and $\tau = 10^{-(2 + p)} \cdot 2^{-\frac{N-32}{32}}$ for $p \in \{1,2\}$ and
$N$ being the number of elements in each axis.}
  \label{tab:harmonic_eoc}
  \begin{tabular}{l @{\hspace{1em}} c @{\hspace{0.8em}} c @{\hspace{0.8em}} c @{\hspace{0.8em}}
c @{\hspace{0.8em}} c @{\hspace{0.8em}} c @{\hspace{0.8em}} c @{\hspace{0.8em}} c
@{\hspace{0.8em}} c @{\hspace{0.8em}} c}
    \hline\hline
    & & \multicolumn{4}{c}{$A = 0.1$} & \multicolumn{4}{c}{$A = 0.6$} \\
    \cline{3-6} \cline{7-10}
    & & \multicolumn{2}{c}{$L^2$ Error} & \multicolumn{2}{c}{$H^1$ Error} & \multicolumn{2}{c}{$L^2$
Error} & \multicolumn{2}{c}{$H^1$ Error} \\
    Scheme, $p$ & $N$ & Error & EOC & Error & EOC & Error & EOC & Error & EOC \\
    \hline
                         &    32  &  1.35e-03 &    --- &  1.01e-01 &    --- &  7.47e-03 &    --- &  6.06e-01 &    --- \\
                         &    64  &  3.46e-04 &   1.96 &  5.04e-02 &   1.01 &  1.91e-03 &   1.97 &  3.02e-01 &   1.00 \\
    \sipgDL, $1$         &   128  &  8.74e-05 &   1.99 &  2.52e-02 &   1.00 &  4.82e-04 &   1.99 &  1.51e-01 &   1.00 \\
                         &   256  &  2.20e-05 &   1.99 &  1.26e-02 &   1.00 &  1.21e-04 &   1.99 &  7.56e-02 &   1.00 \\
    \hline
                         &    32  &  2.19e-05 &    --- &  5.10e-03 &    --- &  1.32e-04 &    --- &  3.06e-02 &    --- \\
                         &    64  &  2.72e-06 &   3.01 &  1.28e-03 &   2.00 &  1.63e-05 &   3.01 &  7.66e-03 &   2.00 \\
    \sipgDL, $2$         &   128  &  3.40e-07 &   3.00 &  3.19e-04 &   2.00 &  2.04e-06 &   3.00 &  1.91e-03 &   2.00 \\
                         &   256  &  4.25e-08 &   3.00 &  7.98e-05 &   2.00 &  2.55e-07 &   3.00 &  4.79e-04 &   2.00 \\
    \hline
                         &    32  &  1.35e-03 &    --- &  1.01e-01 &    --- &  7.48e-03 &    --- &  6.06e-01 &    --- \\
                         &    64  &  3.46e-04 &   1.96 &  5.04e-02 &   1.01 &  1.91e-03 &   1.97 &  3.02e-01 &   1.00 \\
    \sipgL, $1$          &   128  &  8.74e-05 &   1.99 &  2.52e-02 &   1.00 &  4.82e-04 &   1.99 &  1.51e-01 &   1.00 \\
                         &   256  &  2.20e-05 &   1.99 &  1.26e-02 &   1.00 &  1.21e-04 &   1.99 &  7.56e-02 &   1.00 \\
    \hline
                         &    32  &  2.19e-05 &    --- &  5.10e-03 &    --- &  1.32e-04 &    --- &  3.06e-02 &    --- \\
                         &    64  &  2.72e-06 &   3.01 &  1.28e-03 &   2.00 &  1.63e-05 &   3.01 &  7.66e-03 &   2.00 \\
    \sipgL, $2$          &   128  &  3.40e-07 &   3.00 &  3.19e-04 &   2.00 &  2.04e-06 &   3.00 &  1.91e-03 &   2.00 \\
                         &   256  &  4.25e-08 &   3.00 &  7.98e-05 &   2.00 &  2.55e-07 &   3.00 &  4.79e-04 &   2.00 \\
    \hline
                         &    32  &  1.35e-03 &    --- &  1.01e-01 &    --- &  7.47e-03 &    --- &  6.06e-01 &    --- \\
                         &    64  &  3.46e-04 &   1.96 &  5.04e-02 &   1.01 &  1.91e-03 &   1.97 &  3.02e-01 &   1.00 \\
    \swipDL, $1$         &   128  &  8.74e-05 &   1.99 &  2.52e-02 &   1.00 &  4.82e-04 &   1.99 &  1.51e-01 &   1.00 \\
                         &   256  &  2.20e-05 &   1.99 &  1.26e-02 &   1.00 &  1.21e-04 &   1.99 &  7.56e-02 &   1.00 \\
    \hline
                         &    32  &  2.19e-05 &    --- &  5.10e-03 &    --- &  1.32e-04 &    --- &  3.06e-02 &    --- \\
                         &    64  &  2.72e-06 &   3.01 &  1.28e-03 &   2.00 &  1.63e-05 &   3.01 &  7.66e-03 &   2.00 \\
    \swipDL, $2$         &   128  &  3.40e-07 &   3.00 &  3.19e-04 &   2.00 &  2.04e-06 &   3.00 &  1.91e-03 &   2.00 \\
                         &   256  &  4.25e-08 &   3.00 &  7.98e-05 &   2.00 &  2.55e-07 &   3.00 &  4.79e-04 &   2.00 \\
    \hline
                         &    32  &  1.35e-03 &    --- &  1.01e-01 &    --- &  7.48e-03 &    --- &  6.06e-01 &    --- \\
                         &    64  &  3.46e-04 &   1.96 &  5.04e-02 &   1.01 &  1.91e-03 &   1.97 &  3.02e-01 &   1.00 \\
    \swipL, $1$          &   128  &  8.74e-05 &   1.99 &  2.52e-02 &   1.00 &  4.82e-04 &   1.99 &  1.51e-01 &   1.00 \\
                         &   256  &  2.20e-05 &   1.99 &  1.26e-02 &   1.00 &  1.21e-04 &   1.99 &  7.56e-02 &   1.00 \\
    \hline
                         &    32  &  2.19e-05 &    --- &  5.10e-03 &    --- &  1.32e-04 &    --- &  3.06e-02 &    --- \\
                         &    64  &  2.72e-06 &   3.01 &  1.28e-03 &   2.00 &  1.63e-05 &   3.01 &  7.66e-03 &   2.00 \\
    \swipL, $2$          &   128  &  3.40e-07 &   3.00 &  3.19e-04 &   2.00 &  2.04e-06 &   3.00 &  1.91e-03 &   2.00 \\
                         &   256  &  4.25e-08 &   3.00 &  7.98e-05 &   2.00 &  2.55e-07 &   3.00 &  4.79e-04 &   2.00 \\
    \hline\hline
  \end{tabular}
\end{table}
\begin{table}[htbp]
  \renewcommand{\arraystretch}{1.1}
  \centering
  \caption{Ex.~\ref{ex:trig} in 3D: EOC for $L^2$ and $H^1$ errors for $A = 0.1$
and
$\dt = 10^{-3} \cdot 2^{-\frac{N-8}{8}}$.}
  \label{tab:3d_eoc}
  \begin{tabular}{l @{\hspace{1.5em}} l @{\hspace{1.5em}} c @{\hspace{1.5em}} c @{\hspace{1.5em}}
c @{\hspace{1.5em}} c @{\hspace{1.5em}} c}
    \hline\hline
    & & & \multicolumn{2}{c}{$L^2$ Error} & \multicolumn{2}{c}{$H^1$ Error} \\
    Scheme & Order & $N$ & Error & EOC & Error & EOC \\
    \hline
                    &                   &     8 & 5.50e-03 &    --- & 3.48e-01 &    --- \\
                    &                   &    16 & 1.40e-03 &   1.97 & 1.76e-01 &   0.98 \\
    \sipgDL            & $ 1 $             &    32 & 3.51e-04 &   1.99 & 8.76e-02 &   1.01 \\
                    &                   &    64 & 8.80e-05 &   2.00 & 4.37e-02 &   1.00 \\
    \hline
                    &                   &     8 & 5.50e-03 &    --- & 3.48e-01 &    --- \\
                    &                   &    16 & 1.40e-03 &   1.97 & 1.76e-01 &   0.98 \\
    \swipDL            & $ 1 $             &    32 & 3.51e-04 &   1.99 & 8.76e-02 &   1.01 \\
                    &                   &    64 & 8.80e-05 &   2.00 & 4.37e-02 &   1.00 \\
    \hline\hline
  \end{tabular}
\end{table}

In Table~\ref{tab:harmonic_eoc}, we present the $L^2$ and $H^1$ errors and EOC for
Ex.~\ref{ex:trig} for the different schemes with amplitudes $A = 0.1$ and $A = 0.6$
given $\beta = 5$.
The results are consistent with the theoretically expected convergence rates.
Additionally, \swipDL/\sipgDL and \swipL/\sipgL yield identical errors and
EOC, which is consistent with the fact that the mobility flux does not affect the
consistency of the scheme, but only the stability and coercivity. Similar behavior
for \sipgL and \swipL was also observed in \cite{Gunnarsson:2026}. We also tested if the same convergence rates hold in 3D, and the results are presented
in Table~\ref{tab:3d_eoc}.
The observed rates confirm the theoretical predictions.
\par

For the remainder of the numerical experiments, we employ the \swipDL scheme with
$\alpha = \frac{1}{2}$, as it yields a more diffusive flux and consequently improved
consistency. Furthermore, we use $\beta = 3$ for the mobility flux as in Rmk.~\ref{rem:fs}.
\begin{example}[Merging droplets]\label{ex:droplets}
We consider the CH equations without advection,
i.e., $\mathbf{u}(\cdot, \cdot) = \mb{0}$, in the domain $\Omega = [0,1]^2$.
The initial condition is given by a smooth profile
\begin{equation}
    \pf(\mbx,0) = (1 - 10^{-8})\left(2\min\left\{\left(1 + 2^{-1}\sum_{j=1}^{2} \tanh\left(\frac{r
- || \mathbf{x} - \mathbf{c}_j||}{\sqrt{2}Cn}\right)
\right),1\right\} - 1\right),
\end{equation}
where $r = 0.2$ is the droplet radius, with central points $\mathbf{c}_1 = (0.3,
0.5)^T$ and $\mathbf{c}_2 = (0.7, 0.5)^T$, Cahn number $Cn = h_{\min}$, and Peclet
number $Pe^{-1} = 3 Cn$. The simulation is run for $ t \leq T = 0.4$.
\end{example}
\begin{example}[Rotating Merging Bubbles] \label{ex:stationary}
We consider the CHNS equations in the domain $\Omega
= [-0.5,0.5]^2$, with an initial velocity field
\begin{equation}
\mathbf{u}(\mbx, 0) = \chi \left( x_2 \left(0.16 - ||\mbx||^2\right)_{\oplus},
-x_1 \left(0.16 - ||\mbx||^2\right)_{\oplus} \right),
\end{equation}
where $\mbx = (x_1, x_2)^T$, and $\chi = 100$ is a scaling factor. The initial phase-field
profile
is
\begin{equation}
\pf(\mbx, 0) = 0.995\left(2\min\left\{\left(1 + 2^{-1}\sum_{j=1}^{2} \tanh\left(\frac{r_j- || \mathbf{x} - \mathbf{c}_j||}{\sqrt{2}Cn}\right)\right),1\right\} - 1\right),
\end{equation}
where $r_1 = 0.25$ and $r_2 = 0.15$ are the radii of the respective droplets, with
central points $\mathbf{c}_1 = (0.1, 0.1)^T$ and $\mathbf{c}_2 = (-0.15, -0.15)^T$,
respectively. The following non-dimensional numbers are considered: Reynolds number
$Re = 1$, Cahn number $Cn = h_{\min}$, Weber number $We = Cn^{-1}$, viscosities
$\mu_1 = \mu_2 = 1$, densities $\rho_1 = 100$ and $\rho_2 = 1$, and Peclet number
$Pe^{-1} = 3 Cn$. The simulation is run for $ t \leq T = 0.2$.
\end{example}
For Exs.~\ref{ex:droplets} and \ref{ex:stationary}, we consider an $hp$-adaptive
formulation with $h_{\min} = \frac{1}{64}$ and $h_{\max} = \frac{1}{32}$ using
the \swipDL scheme and a prescribed maximum polynomial degree of $\pmax$ and $\pmin \geq 0$. For the time increments, following \cite{Gunnarsson:2026}, we set
$\dt = 10^{-3}$ for Ex.~\ref{ex:droplets} and $\dt = 5\cdot 10^{-4}$ for Ex.~\ref{ex:stationary}.
We simulate with different order levels of $p$ for $p$-adaptivity and also compare
to only the $h$-adaptive formulation, i.e., $p = \pmax$ for all elements $K \in
\grid$. We observe that convergence was not achieved for Ex.~\ref{ex:droplets} with
$p = 2$ for the $h$-adaptive formulation; convergence was only obtained,
in that case,
when $||\pf_h||_{L^\infty(K)} \lessapprox 1$, as was also stressed in
\cite{Gunnarsson:2026} regarding careful selection of the initial data. To the best
of the authors' knowledge there is no clear
proof or indication on how to find a safe initial value.
\begin{figure}
\includegraphics[width=0.999\textwidth]{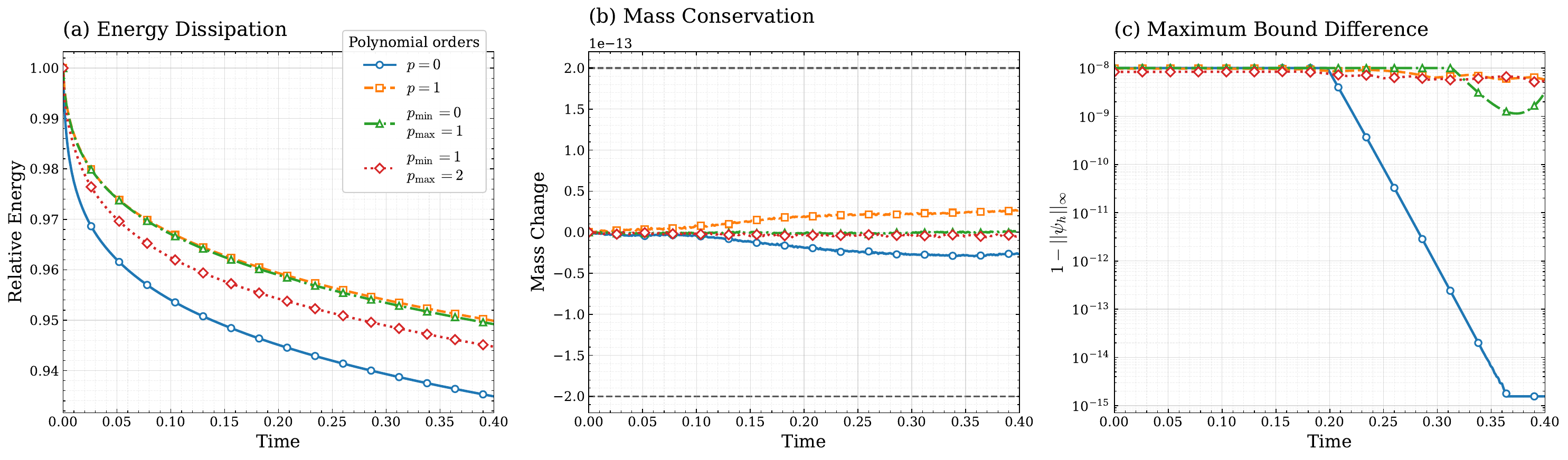}
\caption{Ex.~\ref{ex:droplets}: Energy, mass, and boundedness for \swipDL.} \label{fig:droplets}
\end{figure}
Figure~\ref{fig:droplets} shows the energy, mass, and boundedness for Ex.~\ref{ex:droplets}
using the $hp$-adaptive formulation of \swipDL for various $p$-levels with $hp$-adaptivity. The energy is monotonically non-increasing, mass is conserved, and the phase-field
remains bounded, consistent with the theoretical properties. The results are similar
to those obtained
using only $h$-adaptivity with $p = 1$, and for $p_{\max} = 1$ with $hp$-adaptivity, which suggests
that the $hp$-adaptive formulation preserves accuracy
alongside the $h$-adaptive formulation, while providing a significant reduction
in complexity and degrees of freedom by using higher-order polynomials where needed and lower-order polynomials elsewhere with larger elements.
\begin{figure}[ht]
\includegraphics[width=0.999\textwidth]{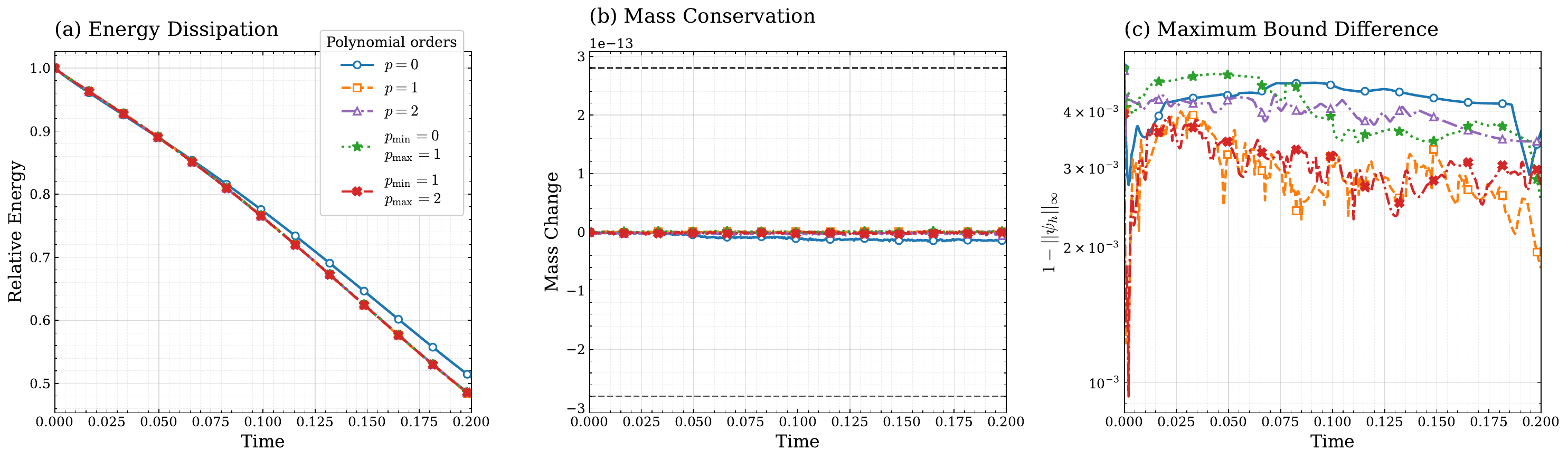}
\caption{Ex.~\ref{ex:stationary}: Energy, mass, and boundedness for \swipDL. Note that the lines for the energy dissipation overlap when $p \neq 0$.
  } \label{fig:stationary}
\end{figure}
Figure~\ref{fig:stationary} shows that the energy is non-increasing, mass is
conserved, and the phase-field remains bounded for Ex.~\ref{ex:stationary}
using the $hp$-adaptive formulation of \swipDL both with and without $p$-adaptivity. The results are similar to those obtained using
only $h$-adaptivity for the studied $p$-levels, which
further indicates that the $hp$-adaptive formulation preserves physical
accuracy alongside the $h$-adaptive formulation, while providing a significant
reduction in complexity and degrees of freedom, similarly to our discussion for Ex.~\ref{ex:droplets}.
For both studied examples, minor mass drifts, as can be found in Figs.~\ref{fig:droplets} and~\ref{fig:stationary}, are a known issue with potential deviations scaling
as the non-linear tolerance as noted in \cite{Gunnarsson:2026}. Regardless, as shown in Figs.~\ref{fig:droplets} and~\ref{fig:stationary}, the mass change is negligible compared to the lines at $\Delta m  = \pm \frac{T}{\dt} \epsilon$ for $\epsilon = 5 \cdot 10^{-16}$ for Ex.~\ref{ex:droplets} and $\epsilon = 7 \cdot 10^{-16}$ for Ex.~\ref{ex:stationary}.

%% file: summary.tex
A new class of DG schemes for the CH equations with degenerate mobility has been proposed, which is provably coercive and stable under a certain condition on the mobility flux for the penalty term. The introduced mobility flux formulations for the penalty term include \sipgDL, which employs the intersection maximum of the mobility, and \swipDL, which uses a harmonic average. These mobility fluxes yield similar estimates for the coercivity condition and accuracy as previous schemes. Numerical experiments illustrate the performance of the proposed schemes, demonstrating that they achieve optimal convergence rates and that the harmonic mobility flux provides potentially significant improvements in stability and accuracy compared to the arithmetic mean. An $hp$-adaptive formulation has also been presented, showing significant reductions in complexity and degrees of freedom compared to the $h$-adaptive formulation while maintaining similar accuracy and physical fidelity.
Future work includes extending the theoretical analysis to provide rigorous structure-preserving proofs for the \swipDL scheme in the discrete setting, exploiting the harmonic average to extend Theorem~\ref{thm:bounded}. Additionally, more challenging test cases, such as the rising bubble benchmark, will further validate the performance of the proposed schemes for coupled CHNS systems and assess adaptivity effects under more demanding conditions.